\documentclass[11pt,twoside]{amsart}
\usepackage[russian,english]{babel}
\usepackage{amssymb}
\usepackage{latexsym}
\catcode`\@=11
\def\omathop#1#2#3{\let\temp=#1\def\letter{#2}
  \ifcat#3_ \let\next\@@olim\else\let\next\@olim\fi\next#3}
\def\@olim{\letter\text{-}\!\temp}
\def\@@olim_#1{\mathchoice{
   \setbox0=\hbox{$\displaystyle\letter\text{-}\!\temp\!\text{-}\letter$}
   \setbox2=\hbox{$\displaystyle\temp$}
   \setbox4=\hbox{$\scriptstyle#1$}
   \dimen@=\wd4 \advance\dimen@ by -\wd2 \divide\dimen@ by2
   \def\next{\letter\text{-}\!\temp_{\hbox to 0pt{\hss$\scriptstyle#1$\hss}}
     \hskip\dimen@}
   \ifdim\wd2>\wd4 \def\next{\@olim_{#1}}\fi
   \ifdim\wd4>\wd0 \def\next{\mathop{\llap{$\letter$-}\!\temp}\limits_{#1}}\fi
   \next}
   {\@olim_{#1}}{\@olim_{#1}}{\@olim_{#1}}}

\def\olim{\omathop{\lim}{o}}

\newcommand{\bl}{\left|}                            
\newcommand{\br}{\right|}                           
\newcommand{\be}{\begin{equation}}
\newcommand{\ee}{\end{equation}}

\def\c{\cite }
\def\d{{\,\rm d\,}}

\def\al{\alpha}

\def\phi{\varphi}
\newcommand{\bi}{\begin{itemize}}
\newcommand{\ei}{\end{itemize}}
\newcommand{\bn}{\begin{enumerate}}
\newcommand{\en}{\end{enumerate}}
\def\R{\Bbb{R}}                                     
\def\N{\mathbb{N}}
\def\cA{\mathcal{A}}
\def\cF{\mathcal{F}}
\def\cU{\mathcal{U}}

\def\fF{\mathfrak{F}}
\def\fP{\mathfrak{P}}
 
\def\eins{\mathbf{1}}

\newcommand{\hide}[1]{}          

\theoremstyle{plain}
\newtheorem{thm}{Theorem}[section]

\newtheorem{prop}[thm]{Proposition}
\newtheorem{cor}[thm]{Corollary}

\newtheorem{example}[thm]{Example}

\theoremstyle{definition}
\newtheorem{definition}[thm]{Definition}

\numberwithin{equation}{section}

\begin{document}
\title[Disjointness and order projections]{Disjointness and order projections in the vector lattices of abstract  Uryson operators}

\author{M.~A.~Pliev and M.~R.~Weber}

\address{South Mathematical Institute of the Russian Academy of Sciences\\
str. Markusa 22 \\
362027 Vladikavkaz,  Russia}
\email{maratpliev@gmail.com}
\address{Technical University Dresden \\
Department of Mathematics, Institut of Analysis \\
01062 Dresden, Germany}
\email{martin.weber@tu-dresden.de}

\begin{abstract}
Projections onto several special subsets 
in the Dedekind complete vector lattice of orthogonally additive, order bounded
(called abstract Uryson) operators between two vector lattices $E$ and $F$ are considered 
and some new formulas are provided. 
\end{abstract}

\maketitle
%
%
\section{Introduction}
%
%
The study of nonlinear maps between vector lattices is a growing area and a subject of intensive investigations 
\cite{BAP,Gum,PP}, where the background has to be found in the nonlinear integral operators, see e.g. \cite{KZPS}.
The interesting class $\mathcal{U}(E,F)$ of nonlinear, order bounded, orthogonally additive operators, 
the so called abstract Uryson operators from a vector lattice $E$ into a  vector lattice $F$ 
was introduced and studied in 1990 by Maz\'{o}n and Segura de Le\'{o}n \cite{Maz-1,Maz-2}, and then
considered to be defined on lattice-normed spaces by Kusraev and Pliev in \cite{Ku-1,Ku-2,Pl-3}. 
If $F$ is Dedekind complete then $\mathcal{U}(E,F)$ turns out to be a Dedekind vector lattice and, so the 
band or order projections are of great interest as a tool for further investigation.    
\par
In this paper some new formulas for projections in $\mathcal{U}(E,F)$ are obtained which are of interest on their own. 
In a forthcoming paper these formulas play an important role in the investigation of finite elements 
in $\mathcal{U}(E,F)$.\footnote{\,The first named author was supported by the Russian Foundation of Fundamental Research, 
the grant number 14-01-91339.} 
%
%
\section{Preliminaries}
%
%
The  goal of this section is to introduce some basic definitions and facts. General information on vector lattices
the reader can find in the books \cite{Al,Ku, Web14, Za}.
\par
Recall that an element $z$ in a vector lattice $E$ is said to be a {\it component} or a
{\it fragment} of $x$ if $z\bot(x-z)$.
The notations $x=y\sqcup z$ and $z\sqsubseteq x$ mean that $x=y+z$ with $y\bot z$ and that $z$ is a fragment
of $x$, respectively.
The set of all fragments of the element $x\in E$ is denoted by  $\cF_x$.
Let be $x\in E$. A collection $(\rho_\xi)_{\xi\in \Xi}$ of elements in $E$ is called a {\it partition of} $x$ if
$\bl\rho_\xi\br\wedge \bl\rho_\eta\br=0$, whenever $\xi\neq \eta$ and $x=\sum_{\xi\in \Xi}\rho_\xi$.
\begin{definition} \label{def:ddmjf0}
Let $E$ be a vector lattice and let $X$ be a real vector space.
An operator $T\colon E\rightarrow X$ is called \textit{orthogonally additive} if $T(x+y)=T(x)+T(y)$
whenever $x,y\in E$ are disjoint elements, i.e. $\bl x\br\wedge \bl y\br =0$.
\end{definition}
It follows from the definition that $T(0)=0$.
It is immediate that the set of all orthogonally additive operators
is a real vector space with respect to the natural linear operations.
\par
So, the orthogonal additivity of an operator $T$ will be expressed as \linebreak $T(x\sqcup y)=T(x)+T(y)$.
\begin{definition}
Let $E$ and $F$ be vector lattices. An orthogonally additive operator $T\colon E\rightarrow F$ is called:
\begin{itemize}
  \item \textit{positive} if $Tx \geq 0$ holds in $F$ for all $x \in E$,
  \item \textit{order bounded} if $T$ maps order bounded sets in $E$ to order bounded sets in $F$.
\end{itemize}
An orthogonally additive order bounded operator $T\colon E\rightarrow F$ is called an \textit{abstract Uryson} operator.
\end{definition}
The set of all abstract Uryson operators from $E$ to $F$ we denote by $\cU(E,F)$.
 If $F=\R$ then an element $f\in \cU(E,\R)$ is called an {\it abstract Uryson functional}.
\par
A positive linear order bounded operator $A\colon E\to F$ defines a positive abstract Uryson operator by means
of $T(x)=A(\bl x\br)$ for each $x\in E$.
\par
 We will consider an important example. 
The most famous examples are the nonlinear integral Uryson operators which are well known and thoroughly studied e.g.
in \cite{KZPS}, chapt.5.
\par\smallskip
Let $(A,\Sigma,\mu)$ and $(B,\Xi,\nu)$ be $\sigma$-finite complete measure spaces, and let $(A\times B,\mu\times\nu)$ denote the completion of their product measure space. 
Let $K:A\times B\times\Bbb{R}\rightarrow\Bbb{R}$ be a function satisfying the following conditions\footnote{\,$(C_{1})$ and $(C_{2})$ are called the Carath\'{e}odory 
conditions.}:
\begin{enumerate}
  \item[$(C_{0})$] $K(s,t,0)=0$ for $\mu\times\nu$-almost all $(s,t)\in  A\times B$;
  \item[$(C_{1})$] $K(\cdot,\cdot,r)$ is $\mu\times\nu$-measurable for all $r\in\Bbb{R}$;
  \item[$(C_{2})$] $K(s,t,\cdot)$ is continuous on $\Bbb{R}$ for $\mu\times\nu$-almost all $(s,t)\in A\times B$.
\end{enumerate}
Denote by $L_0(A,\Sigma,\mu)$ or, shortly by $L_0(\mu)$, the ordered space of all $\mu$-measurable and $\mu$-almost 
everywhere finite functions on A with the order $f\leq g$ defined as $f(s)\leq g(s) \,\mu$-almost everywhere on $A$.
Then $L_0(\mu)$ is a vector lattice.
Analogously the space $L_0(B,\Xi,\nu)$, or shortly $L_0(\nu)$, is defined.
\begin{example}\label{Ex-0}
Given $f\in L_{0}(A,\Sigma,\mu)$ the function $|K(s,\cdot,f(\cdot))|$ is $\mu$-mea\-surable  for $\nu$-almost
all $s\in B$ and $h_{f}(s):=\int_{A}|K(s,t,f(t))|\,d\mu(t)$ is a well defined $\nu$-measurable function.
Since the function $h_{f}$ can be infinite on a set of positive measure, we define
\[
\text{Dom}_{A}(K):=\{f\in L_{0}(\mu)\colon \,h_{f}\in L_{0}(\nu)\}.
\]
Define an operator $T\colon\text{Dom}_{A}(K)\rightarrow L_{0}(\nu)$ by setting
\begin{equation}\label{U}
(Tf)(s):=\int_{A}K(s,t,f(t))\,\d\mu(t)\quad\nu\text{-a.e.}  
\end{equation}
Let $E$ and $F$ be order ideals in $L_{0}(\mu)$ and $L_{0}(\nu)$, respectively and $K$ a function
satisfying the conditions $(C_{0})$\,-\,$(C_{2})$.
Then (\ref{U}) defines an orthogonally additive, in general, not order bounded integral operator acting from
$E$ to $F$ if $E\subseteq \text{Dom}_{A}(K)$ and $T(E)\subseteq F$. The operator $T$ is called
{\rm Uryson (integral) operator}.
\end{example}
\hide{
We consider the vector space $\R^m$ for $m \in \N$ as a vector lattice with the usual coordinate-wise order:
for any $x,y \in \R^m$ we set $x \leq y$ provided $e_i^*(x) \leq e_i^*(y)$ for all $i = 1, \ldots, m$,
where $(e_i^*)_{i=1}^m$ are the coordinate functionals on $\R^m$.
\begin{example} \label{Ex-1}
A map $T\colon \R^n\rightarrow\Bbb{R}^{m}$ belongs to $\mathcal{U}(\Bbb{R}^{n},\Bbb{R}^{m})$ if and only if there are real
functions $T_{i,j}:\R\rightarrow\R$,
$1\leq i\leq m$, $1\leq j\leq n$ satisfying the condition $T_{i,j}(0)=0$ and such that $T_{i,j}([a,b])$ is (order) bounded in $\R^m$
for each order interval $[a,b]\subset \R^n$, where the $i$-th component of the vector $T(x)$ is
calculated by the usual matrix rule, i.e.
$$
\left(T(x)\right)_i = e_i^*\bigl(T(x_{1},\dots,x_{n})\bigr) = \sum_{j=1}^{n}T_{i,j}(x_{j}), \quad i=1,\ldots,m
$$
In this case we write $T=(T_{i,j})$.
\end{example}
}
For more examples of abstract Uryson operators see \c{PP}.
\par
\medskip
In $\cU(E,F)$ the order is introduced as follows:  $S\leq T$ whenever $T-S$ is a positive operator.
Then $\cU(E,F)$ becomes an ordered vector space.
If the vector lattice $F$ is Dedekind complete then $\cU(E,F)$ is a Dedekind complete vector lattice and 
the following generalizations of the well known Riesz-Kantorovich 
formulas for linear regular operators hold (see \cite{Al}, Theorems 1.13 and 1.16).
\begin{thm}[\cite{Maz-1}, Theorem~3.2.] \label{f}
Let $E$ and $F$ be vector lattices with $F$ Dedekind complete. 
Then $\mathcal{U}(E,F)$ is a Dedekind complete vector lattice. 
Moreover, for each $S,T\in \mathcal{U}(E,F)$ and $x\in E$ the following conditions hold
\begin{enumerate}
  \item $(T\vee S)(x)=\sup\{T(y)+S(z):\, x = y \sqcup z\}$.
  \item $(T\wedge S)(x)=\inf\{T(y)+S(z):\, x = y \sqcup z\}$.
  \item $T^{+}(x)=\sup\{Ty:\, y \, \sqsubseteq x\}$.
  \item $T^{-}(x)=-\inf\{Ty: \, y \, \sqsubseteq x\}$.
  \item $\bl T \br(x)=\left(T^+\vee T^-\right)(x) = \sup \{T(y)-T(z): x=y\sqcup z\}$
  \item $|T(x)|\leq|T|(x)$.
  \end{enumerate}
\end{thm}
%
%
\section{Disjoint vectors and order projections in $\mathcal{U}(E,F)$}
%
%
Order projections are an important tool in the study of vector lattices.
There are  some interesting results concerning order projections in the spaces of linear, bilinear  
and orthogonally additive operators in vector lattices
\cite{Al-1,Ko,Ko-1,Pag,Pl-3, Pl-4,Pl-5}. 
The peculiarities of the elements in the Dedekind complete vector lattice of abstract Uryson operators originate further 
 new projection formulas, which encourage, in particular, the study of finite elements (\c{Web14}) in such vector lattices.   
 Recall first some basic notions.
Each band $K$ in a Dedekind complete vector lattice $F$ generates
an order (or band) projection $\rho_K\colon F\to K$ defined for $f\in F$ by $\rho_K(f)=f_1$ if $f=f_1+f_2$
with $f_1\in K$ and  $f_2\in K^\perp$, where $f_1$ and $f_2$ are called the projections of the element $f$ onto
the bands $K$ and $K^\bot$, respectively. 
Denote by $\rho_f$ the band projection onto the principle band $\{f\}^{\bot\bot}$ and
by $\rho_f^\perp$ the corresponding band projection onto $\{f\}^\perp$ in $F$. It is clear that $\rho_f^{\bot\bot}=\rho_f$.
The  following formulas for calculation of the projections (of an element $0\leq g\in F$) onto a band $K$  and
onto a principle band $\{f\}^{\bot\bot}$ are well known and used several times:
\[ \rho_K(g)= \sup\{y\in K\colon 0\leq y\leq g\}\quad\mbox{and} \quad
   \rho_f(g)= \sup\limits_{n\in \N}\{g\wedge (nf)\}.  \]
The set of all band projections (shortly, projectors) in $F$ is denoted by $\fP(F)$.
Under the order $\rho'\leq \rho''$ iff $\rho'\circ \rho''=\rho'$ and the Boolean operations
\[ \rho'\wedge\rho''= \rho'\circ\rho'', \quad \rho'\vee\rho'' = \rho'+\rho''-\rho'\circ\rho'', \quad \rho^*= I_F-\rho,
   \]
where $\rho,\rho',\rho''\in \fP(F)$ and $I_F$ is the identity operator on $F$, the set 
$\fP(F)$ turns out to be a Boolean algebra, i.e. a distributive complemented lattice with zero $0$  
and unity $I_F$ (see \c{Ku}, sect.1.3.5). It is clear that $\rho\circ\rho'=\rho'\circ\rho$ and 
$\rho\circ\rho=\rho$ for any $\rho,\,\rho'\in \fP(F)$.
The following relations among projections onto principal bands  
will be frequently used later on: 
$$ \rho_f\wedge\rho_g= \rho_{f\wedge g}, 
\;\; \rho_f\circ \rho_f^\bot =0, \;\; \rho_g g=g,\;\; \rho_f(f\wedge g)=f\wedge g,\;\; \rho_f(v)+\rho_f^\bot v=v.$$ 
As a representative sample we show the first of them. Take $v\in F$. Then  
\begin{eqnarray*}(\rho_f\wedge \rho_g)v & = & (\rho_f\circ \rho_g)v=\rho_f(\rho_g v)= 
 \rho_f\big(\sup\limits_{n\in \N}\{v\wedge ng\}\big) \\ 
& = & \sup\limits_{m\in \N}\big\{\sup\limits_{n\in \N}\{v\wedge ng\}\wedge mf\}\big\}=
      \sup\limits_m\,\sup\limits_n\{v\wedge \inf\{m,n\}g\wedge f \big\}\\ 
& = & \sup\limits_{k\in \N}\{v\wedge k(f\wedge g)\}=\rho_{f\wedge g}(v).
\end{eqnarray*}  
A {\it partition of unity} is a family of projectors $(\rho_\xi)_{\xi\in \Xi}\subset\fP(F)$ such that
$\rho_\xi\wedge\rho_\eta=0$ for $\xi\neq \eta$ and $\sup\limits_{\xi\in \Xi}\rho_\xi=I_F$.
\par
For proving the subsequent theorem we need the following auxiliary proposition, which was proven by the nonstandard methods.
\begin{prop}[\cite{Ku-Kut}, Proposition~5.2.7.2]\label{op-100}
Let $F$ be a  Dedekind complete vector lattice with a weak order unit\footnote{\,An element $u\in F_+$ is a 
{\it weak order unit} if $\{u\}^{\bot\bot}=F$, i.e. except $0$ there
are no elements in $F$ which are disjoint to $u$.} $u$ and $(x_{\lambda})_{\lambda\in\Lambda}$ 
be an order bounded net in $F$. Then the net  $(x_{\lambda})_{\lambda\in\Lambda}$ order converges to an element $x\in F$ 
if and only if for every $\varepsilon>0$ there exists a partition of unity 
$(\rho_{\lambda})_{\lambda\in\Lambda}$ such that
\[
\rho_{\lambda}|x_{\beta}-x|\leq\varepsilon u, \;\; \beta\geq\lambda.
\]
\end{prop}
\par
\begin{thm}\label{op-10}
Let $E,F$ be vector lattices,  $F$ be  Dedekind complete and let $\cA$ be the set of all weak order
units in $F$.
If the operators $T,S\in\cU_{+}(E,F)$  are disjoint, then for every  $x\in E$, $u\in \cA$ and $\varepsilon>0$
there exists a partition of unity $(\pi_{\alpha})_{\alpha\in\Delta}$ in $\fP(F)$ and a family
$(x_\al)_{\al\in\Delta}$ of fragments of $x$, such that
\[
\pi_\al\big(Tx_\al+S(x-x_\al)\big)\leq\varepsilon u \;\mbox{ for all } \; \alpha\in\Delta.
\]
\end{thm}
\begin{proof}
Take $x\in E$. Denote the set of all pairs $\al=(y,z)\in\cF_x\times\cF_x$ of
mutually disjoint fragments of $x$, such that $y+z=x$,  by $\Delta$.
For any $\al=(y,x-y)\in \Delta$ put $f_\al=Ty+S(x-y)$. 
Due to formula (2) of Theorem \ref{f} the disjointness
of the operators $S$ and $T$ implies $\inf\limits_{\alpha\in\Delta}\{f_{\alpha}\}=0$.
Denote by $\Xi$ the collection of all finite subsets of $\Delta$ ordered as usual by inclusion, 
i.e. $\xi\leq \xi'$ iff $\xi\subset \xi'$. 
Introduce the set $(y_{\xi})_{\xi\in\Xi}$ of all infima of finite many elements of the set
$\{f_\al\colon \al\in \Delta\}$, i.e. if $\xi\in\Xi$ is the finite set  $\xi=\{\al_{\xi_{1}},\dots,\al_{\xi_{n}}\}$,
where $\al_{\xi_k}\in \Delta$ for $k=1,\ldots,n$, then
\[ 
       y_\xi=\bigwedge\limits_{i=1}^n f_{\al_{\xi_i}}
\]
The set $(y_{\xi})_{\xi\in\Xi}$ is downwards directed and $\olim\limits_{\xi\in\Xi}y_{\xi}=0$.
By Proposition~\ref{op-100}, for every $\varepsilon>0$ and $u\in\cA$
there exists a partition of unity $(\rho_{\xi})_{\xi\in\Xi}$ in $\fP(F)$ such that
\[
\rho_{\xi}(y_{\xi})\leq \varepsilon u \;\mbox{ for all }\;  \xi\in\Xi.
\]
In particular, $\rho_\xi(f_\al)<\varepsilon u$ if $\xi=\al$. 
\par
Identify now $F$ with a vector sublattice of the Dedekind complete vector lattice
$C_{\infty}(Q)$ of all extended real valued continuous functions on some extremally disconnected compact 
space $Q$ (more exactly with its image under some vector lattice isomorphism), where the choosen weak order 
unit $u$ is mapped onto the constant function  $\eins$ on $Q$ (see \c{AbrAl}, Theorem 3.35).  
Then the order projections $(\rho_{\xi})_{\xi\in\Xi}$  
(of the above partition of unity) are the multiplication operators in the space $C_{\infty}(Q)$ 
generated by the characteristic functions $\eins_{Q_{\xi}}$, respectively, 
where $Q_{\xi}$ for all $\xi\in\Xi$ are closed-open subsets of $Q$ 
such that $Q=\bigcup\limits_{\xi}Q_{\xi}$ and $Q_{\xi}\cap Q_{\xi'}=\emptyset$ 
for every $\xi,\xi'\in \Xi$, $\xi\neq\xi'$. The supremum $\sup\limits_{\xi\in \Xi} \rho_\xi$ is 
the identity operator $I_F$. 
\\
For $\xi\in\Xi$ and $\al\in \Delta$ define the set  
\[
     A_\al^\xi=\{t\in Q_\xi\colon f_\al(t)< f_\beta(t), \, \beta\in \xi,\, \beta\neq \al \} 
\]
and denote by $\overline{A_\al^\xi}$ its closure in $Q_\xi$ and, consequently in $Q$. 
So $\overline{A_\al^\xi}$  are closed-open subsets of $Q$
for every  $\xi\in\Xi$ and $\al\in \Delta$ and, mutually disjoint if at least one index is different $\al\neq \al'$ 
or $\xi\neq \xi'$. 
Denote by $\rho_\al^\xi$ the multiplication operator
generated by the characteristic function $\eins_{\overline{A_{\alpha}^{\xi}}}$,
i.e.  $\rho_\al^\xi(f)= f\cdot\eins_{\overline{A_\al^\xi}}$ for any function $f\in C_{\infty}(Q)$.
It is clear that $\rho_\al^\xi$ is an order projection in $C_{\infty}(Q)$ and 
$\overline{A_\al^\xi}\subset Q_\xi$ implies $\rho_\al^\xi\leq \rho_\xi$. 
Hence $\rho_\al^\xi(f_\al)\leq \varepsilon u$ for every  $\alpha\in\Delta$ and every $\xi\in\Xi$.
By what has been mentioned above the order projections $\rho_\al^\xi$ 
are mutually disjoint, whenever  $\al\neq \al'$ or $\xi\neq \xi'$. 
Therefore, the order projections $\pi_\al=\sup\limits_{\xi\in\Xi}\rho_\al^\xi$ and 
$\pi_{\al'}=\sup\limits_{\xi\in\Xi}\rho_{\al'}^\xi$ are mutually disjoint as well. 
We show that the supremum of all $\pi_\al$ is the identity operator. 
By assuming the contrary there is an order projection $\gamma$ 
which is disjoint to each projection $\pi_\al$ what causes its disjointness to each 
$\rho_\al^\xi$ and finally, $\gamma$ is disjoint to each $\rho_\xi$. 
This contradicts the fact that $(\rho_\xi)_{\xi\in \Xi}$ is a partition of unity. 
Thus  
$(\pi_\al)_{\al\in\Delta}$ is a partition of unity and
\[
\pi_{\alpha}\big(Tx_{\alpha}+S(x-x_{\alpha})\big)\leq \varepsilon u \; \mbox{ for every } \; \alpha\in\Delta.
\]
\end{proof}
\par
\medskip
Similar to the linear case (see \c{Ku}, sect.3.1.5) we characterize next the disjointness
of two positive abstract Uryson operators.
\begin{thm}\label{op-1}
Let $E,F$ be vector lattices, with $F$ Dedekind complete. 
Two operators $S,T\in\cU_+(E,F)$  are disjoint if and only if for arbitrary $x\in E$ and $\varepsilon>0$
there exist a partition $(\rho_\al)_{\al\in \Delta}$ of unity in $\fP(F)$ and a family 
$(x_\al)_{\al\in \Delta}\subset\cF_x$ such that the inequalities
\be\label{f41}
\rho_{\alpha}Tx_\alpha\leq\varepsilon Tx \quad\mbox{and} \quad
\rho_{\alpha}S(x-x_\al)\leq\varepsilon Sx
\ee
hold for all $\al\in \Delta$.
\end{thm}
\begin{proof}
Let  $S\wedge T=0$. For an arbitrary $x\in E$ consider in $F$ the element
         $u=Sx\wedge Tx+\rho_{Sx\wedge Tx}^{\bot}(Sx+Tx)$.
We may write
\begin{eqnarray*}  
    Sx+Tx & = & Sx\wedge Tx+Sx\vee Tx,\\
Sx\vee Tx & = &v_{1}+v_{2},
\end{eqnarray*} 
for $v_{1}=\rho_{Sx\wedge Tx}(Sx\vee Tx)$, $v_{2}=\rho_{Sx\wedge Tx}^{\bot}(Sx\vee Tx)$ and $v_{1}\bot v_{2}$. 
Notice that $v_{1}\in(Sx\wedge Tx)^{\bot\bot}$. 
If $w\in F$ is an element such that $w\bot u$ then 
\[ 
w\bot (Sx\wedge Tx)\;\mbox{ and } \;w\bot \rho_{Sx\wedge Tx}^{\bot}(Sx+Tx).
\]
The first relation implies $w\bot(Sx\wedge Tx)^{\bot\bot}$ and therefore, $w\bot v_1$.  
Due to $Sx\vee Tx=\rho_{Sx\wedge Tx}(Sx\vee Tx)+\rho_{Sx\wedge Tx}^\bot(Sx\vee Tx)$ the second relation 
implies $w\bot \rho_{Sx\wedge Tx}^\bot(Sx\vee Tx)$, i.e. $w\bot v_2$. 
So we have $w\bot(v_1+v_2)$ what together with $w\bot (Sx\wedge Tx)^{\bot\bot}$ implies  $w\bot (Sx+ Tx)$.
Hence it is shown that $Sx+Tx\in\{u\}^{\bot\bot}$ and $u$ is a weak order unit in $\{Sx+Tx\}^{\bot\bot}$. 
We claim that
\be\label{f41a}
\rho_{Sx}u\leq Sx\; \mbox{ and }\; \rho_{Tx}u\leq Tx.
\ee
For that we establish first the two auxiliary relations   
\begin{equation*}
(a) \;\,\rho_{Sx}\circ\rho_{Sx\wedge Tx}^{\bot}Tx=0 \quad\mbox{and} \quad 
(b) \;\,\rho_{Sx}\circ\rho_{Sx\wedge Tx}^{\bot}(Sx+Tx)=\rho_{Sx\wedge Tx}^{\bot}Sx.
\end{equation*}
For $(a)$ consider  
\begin{eqnarray*}
\rho_{Sx}\circ\rho_{Sx\wedge Tx}^{\bot}Tx & = & \rho_{Sx\wedge Tx}^{\bot}\circ\rho_{Sx}Tx \\
              & = & \rho_{Sx\wedge Tx}^{\bot}\circ\rho_{Sx}\circ\rho_{Tx}Tx  
 = \rho_{Sx\wedge Tx}^{\bot}\circ\rho_{Sx\wedge Tx}Tx=0.
\end{eqnarray*}
Then $(b)$ follows immediately as  
\begin{eqnarray*}
\lefteqn{\rho_{Sx}\circ\rho_{Sx\wedge Tx}^{\bot}(Sx+Tx)} \\   
& = &  \rho_{Sx}\circ\rho_{Sx\wedge Tx}^{\bot}Sx+\rho_{Sx}\circ\rho_{Sx\wedge Tx}^{\bot}Tx  \\ 
& = &  \rho_{Sx\wedge Tx}^{\bot}\circ\rho_{Sx}Sx+\rho_{Sx}\circ\rho_{Sx\wedge Tx}^{\bot}Tx  =  
       \rho_{Sx\wedge Tx}^{\bot}Sx+\rho_{Sx}\circ\rho_{Sx\wedge Tx}^{\bot}Tx  \\ 
& = &  \rho_{Sx\wedge Tx}^{\bot}Sx.
\end{eqnarray*}
Now (\ref{f41a}) can be shown as follows:
\begin{eqnarray*}
 \rho_{Sx}u &  =  & \rho_{Sx}\big(Sx\wedge Tx+\rho_{Sx\wedge Tx}^{\bot}(Sx+Tx)\big)             \\
            &  =  & \rho_{Sx}(Sx\wedge Tx)+\rho_{Sx}\circ\rho_{Sx\wedge Tx}^{\bot}(Sx+Tx) \\
            &  =  & \rho_{Sx\wedge Tx}(Sx\wedge Tx)+\rho_{Sx\wedge Tx}^{\bot}Sx\\
            &\leq & \rho_{Sx\wedge Tx}Sx+\rho_{Sx\wedge Tx}^{\bot}Sx=Sx.
\end{eqnarray*}
The same argument is valid for $\rho_{Tx}u$.
By the disjointness of $T$ and $S$ and in view of Theorem\footnote{\,A weak order unit we need for applying this 
theorem is $u+v$, where $u$, as was already mentioned, is a weak order unit in $\{Sx+Tx\}^{\bot\bot}$ 
and $v$ is some weak order unit in $\{Sx+Tx\}^\bot$.}\,\ref{op-10}, for any $\varepsilon >0$,
there is a partition of unity $(\rho_\al)_{\al\in \Delta}$ in $\fP(F)$ and a
family of fragments  $(x_\al)_{\al\in \Delta}$ of the element  $x$, such that
\[ \rho_\al(Tx_{\alpha}+S(x-x_\al))\leq\varepsilon u\; \mbox{ for all } \al\in \Delta.
\]
Consequently, for any $\al\in \Delta$ one has
\begin{eqnarray*}
    \rho_\al Tx_\al & = & \rho_{Tx}\circ\rho_\al Tx_\al\leq\rho_{Tx}\varepsilon u \leq\varepsilon Tx\quad \mbox{ and }\\
\rho_\al S(x-x_\al) & = & \rho_{Sx}\circ\rho_\al S(x-x_\al)\leq\rho_{Sx}\varepsilon u \leq\varepsilon Sx.
\end{eqnarray*}
Let us prove the converse assertion. Take again an arbitrary element $x\in E$.
According to Theorem \ref{f}(2) we must prove that
\[ 
    (T\wedge S)x=\inf\{Ty+Sz\colon x=y\sqcup z \}=0.
\]
By the assumption, for every  $\varepsilon>0$  there is a partition of unity $(\rho_\al)_{\al\in \Delta}$ in $\fP(F)$ 
and a family $(x_\al)_{\al\in \Delta}\subset\cF_x$ with the properties (\ref{f41}).
So we have
\begin{eqnarray*}
 (T\wedge S)x &   =  & \inf\limits_{y\in \cF_x}\{Ty+Sz\colon x=y\sqcup (x-y)\} \\
              & \leq & \inf\limits_{\al\in \Delta}\{Tx_\al+S(x-x_\al)\} \\
              &  =   & \sup\limits_{\al\in \Delta}\rho_\al\big(\inf\limits_{\al\in\Delta}\{Tx_{\alpha}+S(x-x_{\alpha})\}\big) \\ 
              &  =   & \sup\limits_{\al\in \Delta}\rho_\al\big(Tx_\al+S(x-x_\al)\big) \leq\varepsilon(Tx+Sx) .
\end{eqnarray*}%
Hence $(T\wedge S)x=0$.
\end{proof}
For each set $A\subset\mathcal{U}(E,F)$ we denote by $\pi_{A}$ the projector in $\cU(E,F)$
onto the band $A^{\bot\bot}$ and put $\pi_A^\perp=(\pi_A)^{\bot}$ (the projection onto $A^\bot$).
\par
A set $A\subset\cU_{+}(E,F)$  is called { \it increasing} or {\it  upwards directed} if for arbitrary
$S,T\in A$  there exists a $V\in A$ such that $S,T\leq V$. The next formulas show, in particular,
how to calculate the projection onto a band which is generated by an increasing set.
\begin{thm}\label{op-2}
Let $E,F$ be  vector lattices with $F$ Dedekind complete, $A\subset\cU_{+}(E,F)$ be an increasing set.
Then the following equations hold for
 arbitrary $T\in\mathcal{U}_{+}(E,F)$ and  $x\in E$ 
 \begin{eqnarray}
(\pi_A T)(x) & = &
  \sup\limits_{\varepsilon>0\atop S\in A}\,\inf\limits_{y\in\cF_x \atop\rho\in\fP(F)}\{\rho Ty+\rho^{\bot}Tx\colon\rho S(x-y)
                 \leq  \varepsilon Sx\}, \label{f43} \\
(\pi_A^\perp T)(x) & = & 
  \inf\limits_{\varepsilon>0\atop S\in A}\, \sup\limits_{y\in\cF_x\atop\rho\in\fP(F)}\{\rho Ty\colon \rho Sy
                  \leq\varepsilon Sx\}. \label{f44}
\end{eqnarray}
\end{thm}
\begin{proof}
Both formulas are proved by a similar argument. So, only the second one will be proved.
Denote for some $x\in E$ the right-hand side of  (\ref{f44})  by $\vartheta(T)(x)$.
It is clear that the map $\vartheta(T):E\to F$ is order bounded and positive.
If $x=x_{1}+x_{2}$, $x_{1}\bot x_{2}$, then every fragment of  $y\in\cF_x$ has a representation
   $y=y_1 + y_2 $, where $y_i\in\cF_{x_i}$ and $i=1,2$.
Therefore $\vartheta(T)$ is an orthogonally additive operator.
For $T\in\cU_+(E,F)$ put $\kappa(T)=T-\vartheta(T)$ and prove $\kappa(T)=\pi_{A}T$. 
Since for any $y\in \cF_x$ and $\rho\in \fP(F)$ 
\[Tx=\rho Ty+\rho^\bot Ty+T(x-y)\, \mbox{ implies }\,  Tx-\rho Ty= \rho^\bot Tx+\rho T(x-y) \]
one has
\begin{eqnarray*}
\kappa(T)(x) & = & Tx-\vartheta(T)(x) \\
             & = &  \sup_{\varepsilon>0\atop S\in A}\, \inf\limits_{y\in \cF_x\atop \rho\in \fP(F)}\{Tx-\rho Ty\colon \rho Sy\leq \varepsilon Sx\} \\
             & = &  \sup_{\varepsilon>0\atop S\in A}\, \inf\limits_{y\in \cF_x\atop
            \rho\in \fP(F)}\{\rho^\bot Tx+\rho T(x-y) \colon \rho Sy\leq \varepsilon Sx\},
\end{eqnarray*}
 what may be written as
\[
\kappa(T)(x)=\sup\limits_{\varepsilon>0 \atop S\in A}\,\inf\limits_{y\in \cF_x\atop\rho\in \fP(F)}
             \{\rho^{\bot}Tx+\rho Ty\colon\rho S(x-y)\leq\varepsilon Sx\}.
\]
The order ideal, generated by $A$ is order dense in $A^{\bot\bot}$. Therefore 
a net of operators  $(T_{\gamma})_{\gamma\in\Gamma}\subset\mathcal{U}_{+}(E,F)$  exists which belongs  
to the ideal generated by $A$ such that
\[
T_{\gamma}=\sum\limits_{i=1}^{n(\gamma)}\lambda_{i}S_{i}, \mbox{ where } \,
S_{i}\in A,\,n(\gamma)\in\N,\,\gamma\in\Gamma,\,\lambda_{i}\in\R_+
\]
and $T_{\gamma}\uparrow\pi_{A}T$ (\cite{Al}, Theorems~3.3 and 3.4). 
Using the fact that $A$ is an increasing set one has
\[
(T_{\gamma})_{\gamma\in \Gamma}\subset\bigcup\limits_{S\in A \atop n\in\N}[0,nS] .
\]
Fix  $\gamma_{0}\in\Gamma$.
Then  $T_{\gamma_{0}}\leq nS_{0}$ for some  $S_0\in A$ and $n\in\N$.
For arbitrary  $\varepsilon>0$ there exist 
  $\rho\in\fP(F)$  and  $y\in\cF_x$
such\footnote{\,For example, $y=x$ and arbitrary $\rho\in \fP(F)$.} that   $\rho S_0(x-y)\leq\varepsilon S_0x$.
Hence for those $y$ and $\rho$, and due to  $T_{\gamma_0}\leq \pi_AT\leq T$, 
one gets
\begin{eqnarray*}
T_{\gamma_{0}}x & \leq & \rho T_{\gamma_{0}}(x-y)+\rho Ty+\rho^{\bot}Tx\\
                & \leq & \rho nS_0(x-y)+\rho Ty+\rho^{\bot}Tx\leq
\varepsilon nS_0x+\rho Ty+\rho^{\bot}Tx.
\end{eqnarray*}
So by passing first to the infimum (with respect to $y\in \cF_x$ and $\rho\in \fP(F)$) 
and subsequently, to the supremum (with respect to $\varepsilon>0$ and  $S\in A$) 
on the right-hand side\footnote{\,in both cases without touching 
the term $\varepsilon nS_0x$.}  of the last inequalities one obtains 
\begin{eqnarray*}
T_{\gamma_{0}}x & \leq & \varepsilon nS_0x+\inf\limits_{y\in\cF_x\atop \rho\in\fP(F)}
                         \{\rho Ty+\rho^\bot Tx\colon \rho S_0(x-y)
                         \leq\varepsilon S_0x\}\\
                & \leq & \varepsilon nS_0x+\sup\limits_{\varepsilon>0\atop S\in A}
                         \inf\limits_{y\in\cF_x\atop \rho\in\fP(F)}
                         \{\rho Ty+\rho^\bot Tx\colon \rho S(x-y)
                         \leq\varepsilon Sx\}\\         
                & \leq & \varepsilon nS_0x+\kappa(T)(x).
\end{eqnarray*}
Since $\varepsilon$ is arbitrary thus $T_{\gamma_{0}}x\leq\kappa(T)(x)$ is proved.
Consequently  $\sup_{\gamma\in\Gamma}T_\gamma x\leq\kappa(T)(x)$
and
\[
(\pi_{A}T)x = \sup\limits_{\gamma\in \Gamma}T_{\gamma}x\leq\kappa(T)(x)\;\mbox{ implies }\;
              \vartheta(T)(x)\leq(\pi_{A}^\bot T)x.
\]
Since  $x\in E$ is arbitrary one has  $\vartheta(T)\leq\pi_{A}^\perp T$.
\par
The converse inequality is proved as follows.
For arbitrary  $T\in\mathcal{U}_{+}(E,F)$ the following holds
\[
\vartheta(\pi_A^\bot T)\leq\vartheta(T) =  \vartheta(\pi_{A}^\bot T)+\vartheta(\pi_{A}T).
\]
On  the  other  hand,  by  what  has been  proved, one has
\[
\vartheta(\pi_{A}T)\leq\pi_{A}^\bot\pi_{A}T=0.
\]
Therefore  $\vartheta(T)=\vartheta(\pi_{A}^\bot T)$ and 
$\kappa T=T-\vartheta T=\pi_AT+\pi^\bot_AT-\vartheta(\pi^\bot_AT).
$
Followingly, in order to conclude $\kappa(T)=\pi_AT$ it remains to show that
$\vartheta(\pi_{A}^\bot T)=\pi_{A}^\bot T$. 
So, let be $C=\pi_{A}^\bot T$ and $S\in A$. 
Then $C\geq 0$ and $C\wedge S=0$. If one shows $\kappa(C)=0$,  
then $\kappa(\pi_A^\bot T)=\pi^\bot_AT-\vartheta(\pi^\bot_AT)$ implies $\pi^\bot_AT=\vartheta(\pi^\bot_AT)$. 
\par
By Theorem~\ref{op-1} for any $\varepsilon>0$ and $x\in E$ there exist a partition of unity  
$(\rho_{\alpha})_{\alpha\in\Delta}$ in $\fP(F)$ and a family    
$(x_\al)_{\al\in\Delta}\subset \cF_x$ such that
\[
\rho_\al Cx_\al    \leq\varepsilon Cx \, \mbox{  and }\,
\rho_\al S(x-x_\al)\leq\varepsilon Sx.
\]
Then\footnote{\,by using the relations $\rho\circ\rho=\rho$ and $\rho\circ\rho^\bot=0$.} 
one has
\begin{eqnarray*}
\varepsilon Cx & \geq & \rho_\al Cx_\al=\rho_\al(\rho_\al Cx_\al+\rho_\al^{\bot}Cx)\\
               & \geq & \rho_\al\inf\limits_{y\in\cF_{x}\atop \rho\in\fP(F)}
               \{\rho Cy+\rho^{\bot}Cx\colon \rho S(x-y)\leq\varepsilon Sx\}, 
\end{eqnarray*}
what implies   
\begin{eqnarray*}
\varepsilon Cx & \geq & \sup\limits_{\al\in \Delta}\rho_\al \inf\limits_{y\in\cF_{x}\atop \rho\in\fP(F)}
                        \{\rho Cy+\rho^{\bot}Cx\colon \rho S(x-y)\leq\varepsilon Sx\} \\
               &   =  & \inf\limits_{y\in\cF_{x}\atop \rho\in\fP(F)}
               \{\rho Cy+\rho^{\bot}Cx\colon \rho S(x-y)\leq\varepsilon Sx\}.
\end{eqnarray*}
Observe that the left-hand side of the inequality does not depend on $S$ and observe that the expression on the right-haand side  
increases if $\varepsilon$  decreases. For fixed $\varepsilon_0>0$ one has then 
\[   \varepsilon_0 Cx  \geq \inf\limits_{y\in\cF_{x}\atop \rho\in\fP(F)}
                            \{\rho Cy+\rho^{\bot}Cx\colon \rho S(x-y)\leq\varepsilon_0 Sx\}
\]
and for  $0<\varepsilon'<\varepsilon_0$ 
\begin{eqnarray*}
\varepsilon_0Cx \geq \varepsilon'Cx &\geq & \inf\limits_{y\in\cF_{x}\atop \rho\in\fP(F)}
                                            \{\rho Cy+\rho^{\bot}Cx\colon \rho S(x-y)\leq\varepsilon' Sx\} \\
                                    &\geq & \inf\limits_{y\in\cF_{x}\atop \rho\in\fP(F)}
                            \{\rho Cy+\rho^{\bot}Cx\colon \rho S(x-y)\leq\varepsilon_0 Sx\}.
\end{eqnarray*}
It follows 
\[ \varepsilon_0Cx \geq \sup\limits_{\varepsilon'>0\atop S\in A}\inf\limits_{y\in\cF_{x}\atop \rho\in\fP(F)}
                                            \{\rho Cy+\rho^{\bot}Cx\colon \rho S(x-y)\leq\varepsilon' Sx\} 
\] 
and so, $\varepsilon_0 Cx\geq\kappa(C)$ for 
arbitrary $\varepsilon_0>0$. This means $\kappa(C)=0$.
\end{proof}
For the projections onto the principal bands in $\mathcal{U}_{+}(E,F)$ the following formulas are 
obtained as special cases from (\ref{f43}) and (\ref{f44}).  
\begin{cor}\label{c1}
Let  $E,F$ be the same as in the Theorem~\ref{op-2}.
Then for arbitrary $S,T\in\mathcal{U}_{+}(E,F)$ and  $x\in E$ the following formulas\footnote{\,For 
$S\in\mathcal{U}_{+}(E,F)$ the projections onto the bands $\{S\}^{\bot\bot}$ and $\{S\}^\bot$ are  
denoted by $\pi_S$ and $\pi_S^\bot$, respectively.}  hold
\begin{eqnarray*}
    (\pi_{S}T)x & = & \sup\limits_{\varepsilon>0}\inf_{y\in\cF_x \atop\rho\in\fP(F)}\{\rho Ty+
                      \rho^{\bot}Tx \colon \rho S(x-y)\leq\varepsilon Sx\} \\
(\pi_S^\bot T)x & = & \inf\limits_{\varepsilon>0}\sup_{y\in\cF_x \atop\rho\in\fP(F)}\{\rho Ty\colon 
                      \rho Sy\leq\varepsilon Sx\}.
\end{eqnarray*}
\end{cor}
Another formula for $(\pi_S^\bot T)x$ is obtained as follows. 
First, notice that the equality \mbox{$\rho_{Sx}^{\bot}((\pi_S^\bot T)x)=\rho_{Sx}^{\bot}(Tx)$} holds. 
             Indeed, $\rho_{Sx}^\bot$ belongs to $\fP(F)$ and  \linebreak $\rho_{Sx}^{\bot}S(x-0)\leq \varepsilon Sx$ 
             for every $\varepsilon>0$. Therefore 
\begin{eqnarray*}
    \rho_{Sx}^{\bot}(\pi_{S}T)x &=& \sup\limits_{\varepsilon >0}\inf\limits_{y\in \fF \atop \rho\in \fP(F)}
             \{\rho^\bot_{Sx}(\rho Ty+\rho^\bot Tx)\colon \rho S(x-y)\leq \varepsilon Sx\} \\ 
\mbox{\scriptsize(for $y=0$ and  $\rho=\rho_{Sx}^\bot$)}    &\leq &  \sup\limits_{\varepsilon>0}\{\rho^\bot_{Sx}
                       \circ(\rho_{Sx}^\bot (T0)+\rho^{\bot}_{Sx}(Tx))
             \colon \rho_{Sx}^\bot S(x)\leq \varepsilon Sx\}\\
         &= & \rho_{Sx}^{\bot}\circ(\rho_{Sx}^{\bot}(T0)+
                      \rho_{Sx}^{\bot}(Tx)) 
  = \rho_{Sx}^{\bot}\circ\rho_{Sx}(Tx)=0, 
\end{eqnarray*}
i.e. $\rho_{Sx}^\bot(\pi_ST)x=0$. 
Second, for every element $\rho'(Ty)$ with $\rho'\in \fP(F)$ and $y\in\mathcal{F}_{x}$ one has 
\begin{eqnarray*}
(\rho_{Sx}\circ\rho')Ty=(\rho_{Sx}\wedge \rho')Ty =  \rho(Ty),
\end{eqnarray*}
where $\rho\in [0,\rho_{Sx}]$. Then, in particular,
\begin{eqnarray}\label{f45}
(\pi_S^\bot T)x & = & (\rho_{Sx}^{\bot}+\rho_{Sx})(\pi_{S}^\bot T)x \\
                & = & \rho_{Sx}^{\bot}(Tx)+\inf\limits_{\varepsilon>0}\,
                \sup\limits_{y\in\cF_x\atop\rho\in[0,\rho_{Sx}]}\{\rho(Ty)\colon \rho(Sy)\leq\varepsilon Sx\}
                \nonumber .
                \end{eqnarray}
\par
Let  $E,F$ be vector lattices. Fix $\varphi\in\cU(E,\Bbb{R})$ and $u\in F$.
The one-di\-men\-sional (rank-one) abstract  Uryson operator
$\varphi\otimes u\colon E\to F$ is defined as $(\varphi\otimes u)x=\varphi(x)u$.
The projections onto the band $\{\varphi\otimes u\}^{\bot\bot}$ and onto its orthogonal complement
are special cases of Corollary \ref{c1} and can be calculated as follows.
\begin{prop}\label{op-21}
Let  $E,F$ be the same as in the Theorem~\ref{op-2}.
Let  $\varphi\otimes u$ be a  positive abstract rank-one  Uryson operator,
where $\varphi\in\cU_{+}(E,\R)$, $u\in F_{+}$ and $T\in\cU_{+}(E,F)$.
Then for  $x\in E$ the following formulas hold
\begin{eqnarray}
    (\pi_{\varphi\otimes u}T)x & = & \sup\limits_{\varepsilon>0}\,\inf\limits_{y\in\cF_x}
                      \{\rho_u(Ty) \colon \varphi(x-y)\leq\varepsilon \varphi(x)\}\label{f46}\\
(\pi_{\varphi\otimes u}^\bot T)x & = & \rho_u^{\bot}(Tx)+\inf_{\varepsilon>0}\,\sup\limits_{y\in\cF_x}
                       \{\rho_u (Ty)\colon \varphi(y)\leq\varepsilon\varphi(x)\}\label{f47}.
\end{eqnarray}
\end{prop}
\begin{proof}
It is sufficient to prove only formula (\ref{f47}).
For $\rho\in[0,\rho_{u}]$ the expres\-sion on the right side of (\ref{f45}) is
$\varphi(y)\rho u\leq\varepsilon\varphi(x)u$.
Due to 
\[  
      \sup\limits_{0\leq \rho\leq \rho_u}\phi(y)\rho u=\phi(y)\rho_u u=\phi(y)u 
\]
 the last inequality is equivalent
to $\varphi(y)\leq\varepsilon\varphi(x)$.
Observe that for $x\in E$ with $\varphi(x)> 0$ one has
$\rho_{(\phi\otimes u)x}(y)=\sup\limits_{n\in \N}\{y\wedge n\phi(x)u\}=\rho_u(y)$ for any $y\in F$, i.e.
$\rho_{(\varphi\otimes u)x}=\rho_{u}$.
Therefore,
\[\rho\left((\phi\otimes u)y\right)=\phi(y)\rho u\leq 
\varepsilon \phi(x) u=\varepsilon (\phi\otimes u)x,
  \]
i.e. (if for $y\in \cF_x$ the element $x$ is taken) $\rho ((\varphi\otimes u)x)\leq\varepsilon (\varphi\otimes u)x$.
Thus  the  supremum on the right side of the formula (\ref{f45})  is attained at $\rho=\rho_{u}$.
Now the conclusion is exactly the formula (\ref{f47})
\end{proof}
\begin{cor}
Let $E$ be a vector lattice and $T,\varphi\in\cU(E,\R)$. Then the following formula holds
\be\label{f48}
(\pi_{\varphi}T)x=\sup\limits_{\varepsilon>0}\,
           \inf\limits_{y\in\cF_x}\{Ty\colon\varphi(x-y)\leq\varepsilon \varphi(x)\}.
\ee
\end{cor}
This immediately follows  from (\ref{f46}) if $\phi$ is written as $\phi\otimes 1$.
\vspace{2cm}

\end{document}